\documentclass[11pt]{article}
\usepackage{amssymb,amsmath,amsthm}
\usepackage{epsfig}
\usepackage{url}

\usepackage[perpage,symbol*]{footmisc} %for multiline footnote to look nice
\usepackage{hyperref}

\newtheorem{thm}{Theorem}%[section]

\newtheorem{lem}[thm]{Lemma}

\newtheorem{claim}[thm]{Claim}
\newtheorem{conj}[thm]{Conjecture}

\newtheorem{obs}[thm]{Observation}

\theoremstyle{definition}
\newtheorem{remark}[thm]{Remark}

\usepackage{indentfirst}
\usepackage{xspace} % A makrok utani space korrekt kezelese miatt.

\def\Z{\mbox{\ensuremath{\mathbb Z}}\xspace}

\mathcode`l="8000
\begingroup
\makeatletter
\lccode`\~=`\l
\DeclareMathSymbol{\lsb@l}{\mathalpha}{letters}{`l}
\lowercase{\gdef~{\ifnum\the\mathgroup=\m@ne \ell \else \lsb@l \fi}}%
\endgroup

\def\agame{Alpern's Caching Game\xspace}
\def\sgame{Single Caching Game\xspace}
\def\mgame{Multiple Caching Game\xspace}

\def\games{All or Nothing Caching Games with Bounded Queries\xspace}

\begin{document}

\title{\games}
\author{D\"om\"ot\"or P\'alv\"olgyi\thanks{Department of Pure Mathematics and Mathematical Statistics, University of Cambridge.
Research supported by the Marie Sk\l odowska-Curie action of the EU, under grant IF 660400.}
}
%\date{}
\maketitle
\stepcounter{footnote}

\begin{abstract}
We determine the value of some search games where our goal is to find all of some hidden treasures using queries of bounded size.
The answer to a query is either empty, in which case we lose, or a location, which contains a treasure.
We prove that if we need to find $d$ treasures at $n$ possible locations with queries of size at most $k$, then our chance of winning is $\frac{k^d}{\binom nd}$ if each treasure is at a different location and $\frac{k^d}{\binom{n+d-1}d}$ if each location might hide several treasures for large enough $n$.
Our work builds on some results by Cs\'oka who has studied a continuous version of this problem, known as \agame;
we also prove that the value of \agame is $\frac{k^d}{\binom{n+d-1}d}$ for integer $k$ and large enough $n$.
\end{abstract}

\section{Introduction}\label{sec:intro}

We consider the following game theoretic search problem.
Initially a player, called {\em hider}, hides $d$ treasures behind $n$ doors and after the hiding is over, another player, {\em searcher}, needs to find them all.
In each round searcher selects (guesses) at most $k$ doors.
If none of them has a treasure, she\footnote{For simplicity, we use {\em she} for searcher and {\em he} for hider.} loses.
Otherwise, hider reveals one of the treasures behind one of the selected doors.
If she finds all $d$ treasures with a total of $d$ guesses, she wins.

The value of the game is the chance of searcher winning if both players play optimally.
If hider is allowed to hide more treasures behind the same door (\mgame), we denote this value by $v_M(n,d,k)$, and if hider can hide at most one treasure behind each door (\sgame), we denote this value by $v_1(n,d,k)$.
Note that $v_M(n,d,k)$ decreases in $n$ and in $d$ and increases in $k$, while $v_1(n,d,k)$ is not monotone in $d$.
Our main result is to determine both values if $n$ is large enough.

\begin{thm}\label{thm:s} The value of the \sgame is $v_1(n,d,k)=\frac{k^d}{\binom nd}$ if $n\ge dk$.
\end{thm}

\begin{thm}\label{thm:m} The value of the \mgame is $v_M(n,d,k)=\frac{k^d}{\binom{n+d-1}d}$ if $n$ is large enough (compared to $k$ and $d$).
\end{thm}

The upper bounds follow from a simple counting argument and can be formulated in the following meta-lemma, whose continuous equivalent for some games was stated by Cs\'oka \cite{Csoka}.

\begin{lem}\label{lem:upper} The value of any finite search game where searcher needs to find all hidden treasures is bounded from above by the maximal number of hiding possibilities a deterministic strategy of the searcher can reveal divided by the total number of hiding possibilities.
\end{lem}
\begin{proof} If the treasures are hidden according to a fixed distribution, then it is enough to consider deterministic strategies of the searcher, as finite search games have a value.\footnote{This holds even if the outcomes of the queries depend on any, possibly random parameter that is independent of the searcher's strategy.} 
The performance of these is bounded from above with the claimed quantity against uniformly choosing from the available hiding possibilities.
\end{proof}

The proof of our lower bounds also builds on results by Cs\'oka \cite{Csoka}, who has studied a continuous version of the \mgame, known as \agame \cite{Alpernpaper}; see also \cite{Alpernbook}. %\footnote{Note that they use different variables for notation, we have altered them to be more consistent with search theory.}
In this game hider is allowed to dig holes at $n$ possible places in a total of $1$ unit depth, and then hide the $d$ treasures while burying the holes back at any depth.
For example, he can dig one hole $0.8$ deep and another hole $0.2$ deep, then hide a treasure in the first hole at depths $0.7$ and another treasure in the same hole at depth $0.5$, then a third treasure in the second hole at depth $0.2$.
After the hiding of the treasures is over, searcher can dig out a total of $k$ depth and needs to find all $d$ treasures.
(She notices when she finds a treasure.)

Our motivation to introduce the \mgame, which was to some extent already done implicitly in \cite{Csoka}, was the following connection to \agame.

\begin{obs}\label{obs:link} The value of \agame is at least the value of the \mgame if $k$ is an integer, i.e., $v_A(n,d,k)\ge v_M(n,d,k)$.
\end{obs}
\begin{proof} If $k$ is an integer, then searcher can pretend that she is playing the \mgame instead of \agame by always digging simultaneously in $k$ holes with the same speed.
In case there are treasures in more than one hole, in \agame the depth determines which she finds first; this is better for her than the rule of the \mgame, according to which hider decides which treasure she finds.
If she manages to find all $k$ treasures, then she has to dig for at most $k$ depth in total.
Therefore if she wins the \mgame, she also wins \agame with this strategy.
\end{proof}

From Observation \ref{obs:link} and the lower bound of Theorem \ref{thm:m}, combined with the upper bound given by (Cs\'oka's continuous version of) Lemma \ref{lem:upper}, we get the following.

\begin{thm}\label{thm:a} The value of \agame is $v_A(n,d,k)= \frac{k^d}{\binom{n+d-1}d}$ if $k$ is an integer and $n$ is large enough (compared to $k$ and $d$).
\end{thm}

\begin{remark} Observation \ref{obs:link} would also hold for the following variant of the \mgame.
At the beginning we start with $P=k$ {\em power} and in each turn we can guess at most $P$ doors.
If we guess $l<P$ doors and none of them contains a treasure, then instead of losing immediately, we can continue the game with $P:=P-l$ power from now on.
If we denote the value of this new variant by $v_P$, then from the proof of Observation \ref{obs:link} we get $v_A(n,d,k)\ge v_P(n,d,k)\ge v_M(n,d,k)$ and from Theorem \ref{thm:a} that all values are equal for large $n$'s.
\end{remark}

\begin{remark} We have introduced the rule that if there are treasures behind more than one guessed door, it is hider who reveals one of them, so that the argument used in the proof of Observation \ref{obs:link} works.
The game where we use instead the rule that a treasure is revealed from such doors randomly, chosen uniformly either according to the number of doors or to the number of treasures they hide, we get other natural variants of the problem.
The values of these games are bounded from below by the value of the \mgame, $v_M$, and the upper bounds given by Lemma \ref{lem:upper} are also identical, thus the values of these random variants can only differ from $v_M$ for small $n$'s.
\end{remark}

The organization of the rest of the paper is as follows.
In Section \ref{sec:s} we give some simple examples and give the (quite elementary) proof of Theorem \ref{thm:s}.
In Section \ref{sec:m} we prove our main result, Theorem \ref{thm:m}.
Section \ref{sec:small} discusses what happens for small values of $n$ and Section \ref{sec:remarks} contains some further remarks and open problems.
We end the Introduction with a very brief summary of previous results.

\subsection{History of \agame}\label{sec:a}

The value of \agame, $v_A(n,d,k)$, was determined for some small values in \cite{Alpernpaper,Csoka,CsL}.
For example, it is easy to see that $v_A(n,1,k)=\frac{\lfloor k\rfloor}n$, but the problem is open in general already for $d=2$.
Cs\'oka \cite{Csoka} has shown that the uniform hiding strategy guarantees $v_A(n,d,k)\le \frac{k^d}{\binom{n+d-1}d}$ using a  continuous version of Lemma \ref{lem:upper}.
He has also proved that $v_A(n,2,k)= \frac{k^2}{\binom{n+1}2}$ whenever $k$ is an integer and made the following two conjectures.

\begin{conj}[Cs\'oka \cite{Csoka}]\label{conj:weak} $v_A(n,d,k)= \frac{k^d}{\binom{n+d-1}d}$ if $k$ is an integer and $n\ge dk$.
\end{conj}

\begin{conj}[Cs\'oka \cite{Csoka}]\label{conj:strong} $v_A(n,d,k)= \frac{k^d}{\binom{n+d-1}d}$ if $k\ge 1+\frac2{d-1}$ and $n\ge dk$.
\end{conj}

Notice that Conjecture \ref{conj:strong} is stronger than Conjecture \ref{conj:weak} (discounting some small, known cases).
Cs\'oka gave a construction that shows why $k\ge 1+\frac2{d-1}$ is needed and suggested a searching strategy type that requires $n\ge dk$; our strategy for searcher also falls in this category.
This searching strategy type was inspired by another interesting variant of \agame, where instead of $n$ doors the hider can hide the treasures anywhere under a measurable interval of length $n$; this version was introduced and solved almost completely by Cs\'oka \cite{Csoka}, proving that its value equals $\frac{k^d}{\binom{n+d-1}d}$ if $n\ge dk$.

Theorem \ref{thm:a} solves Conjecture \ref{conj:weak} if $n$ is large enough.
We conjecture that in fact Conjecture \ref{conj:weak} and Theorem \ref{thm:m} already hold for $n=dk-1$.
For non-integer $k$ our results have no implication and thus Conjecture \ref{conj:strong} remains wide open, possibly also true already for all $n\ge dk-1$.
This slightly lower bound was conjectured by Cs\'oka for $d=2$, but the answer is known for no $3<k\notin \Z$

\section{Warm-up and \sgame}\label{sec:s}

As a warm-up, let us consider the \sgame where searcher needs to find $d=2$ treasures in $n=4$ doors by guessing $k=2$ doors in each round.
After an initial guess of two doors, say, doors $1$ and $2$, she either loses immediately, or finds a treasure behind one of the doors, say, door $1$.
The other treasure can be behind doors $2$, $3$ or $4$.
Searcher has two options for her second (and final) guess: she can either guess the two doors that she hasn't guessed before, doors $3$ and $4$, or guess door $2$ and one of $3$ or $4$.
What should she do?

This situation is very similar to popular paradoxes, such as Bertrand's box paradox, the Three Prisoners problem and especially to the Monty Hall problem.
In the latter problem, known from a TV show hosted by Monty Hall, contestants were presented three doors, one of which hid a car and the other two goats.
Their goal was to open the door that hid the car.
The contestant could pick a door first (without opening it).
Then the host has opened one of the other two doors, one that hid a goat.
Finally, the contestant could decide whether to open the door she has originally picked, or the other unopened door.
Many contestants opened their originally picked door, giving them a chance of $\frac 13$ to walk away with the car, opposed to $\frac 23$ had they decided to switch.

For illustration, let us calculate the chances of winning in the \sgame for searcher when the treasures are distributed uniformly at random, i.e., each of the $\binom 42$ options has $\frac 16$ probability.
If in the first round she picks doors $1$ and $2$, and in the second round doors $3$ and $4$, then she wins in $4$ out of the $6$ cases, i.e., with probability $\frac 23$.
If in the first round she picks doors $1$ and $2$, and in the second round doors $2$ and $3$ (supposing that she has found a treasure behind door $1$), then she wins in $3$ out of the $6$ cases, i.e., with probability $\frac 12$ (assuming that hider reveals both treasures with $50\%$ chance if searcher guesses both with her first guess).
This shows that in the \sgame it is also better to switch.

\begin{proof}[Proof of Theorem \ref{thm:s}]
Though the upper bound follows from Lemma \ref{lem:upper}, we present it here for this special version in more detail.
In this game hider doesn't have much options, the best he can do is to hide the $d$ treasures uniformly at random.
We need to prove that searcher has at most $\frac{k^d}{\binom nd}$ probability of winning against this distribution of the treasures.
This is enough to prove for deterministic strategies of the searcher.
Suppose that she picks doors $X_i$ in round $i$ (where $X_i$ might depend on the answer to all $X_j$ for $j<i$). %, never picking a door where she has already found a treasure.
To win, she needs each of the $X_i$ to contain a treasure.
(But each $X_i$ containing a treasure doesn't imply that she wins!)
There are $\Pi_i |X_i|\le k^d$ possibilities for the outcomes.
As in total there are $\binom nd$ possibilities, her chance of winning is at most $\frac{k^d}{\binom nd}$ against the uniform distribution.
This can also be attained if she always picks $k$ new doors at random, which can be done if $n\ge dk$, proving the lower bound.
\end{proof}

\section{\mgame}\label{sec:m}

This section contains the proof of our main result about the \mgame. % where hider can door more treasures at the same door.

\begin{proof}[Proof of Theorem \ref{thm:m}]
The upper bound follows from Lemma \ref{lem:upper}; searcher can cover at most $k^d$ of the $\binom{n+d-1}d$ possibilities, and thus her chance of winning is at most $\frac{k^d}{\binom{n+d-1}d}$.
(Alternatively, it also follows from Cs\'oka's upper bound for \agame and Observation \ref{obs:link}.)

The lower bound is more involved.
First, notice that if $k=1$, then $v_M(n,d,1)\ge \frac{1}{\binom{n+d-1}d}$ can be achieved with the following strategy:
Searcher initially picks an allocation of the treasures, $\mu$, uniformly at random from the $\binom{n+d-1}d$ possibilities, then in each round she selects a door with the aim of finding the treasures where they are in $\mu$.

We will use this strategy for $k=1$ to give a strategy for $k>1$.
To this end, suppose that in the above strategy for $k=1$, searcher additionally follows the below rules.
\begin{itemize}
\item If she guesses a door which hides more treasures according to $\mu$, then she keeps on guessing the same door in the subsequent rounds until she finds all the treasures behind this door.
\item When she needs to guess a new door, she always guesses the one which hides the most number of treasures according to $\mu$ among the doors not yet guessed; in case of equality, she picks randomly from the doors with the most number of treasures hidden.
\end{itemize}

For example, suppose that she has picked allocation $\mu$, where doors $1$ and $2$ each have two treasures, door $3$ has none and door $4$ has three treasures.
In this case, in the first three rounds she would pick door $4$, then in rounds $4$ and $5$ she can pick door $1$ (or $2$, respectively, with $50\%$ chance) and finally in rounds $6$ and $7$ door $2$ (or $1$, respectively).
Note that if we were to depict the treasures found at subsequently guessed doors, we would obtain a monotone decreasing sequence, in other words, a Young diagram.
We denote the (unique) Young diagram with only one element by ``$.$'' in subscript.

Still assuming $k=1$, suppose that so far the searcher has already found some treasures behind some doors using the above strategy, giving some Young diagram $\lambda$.
For an outside observer who doesn't know $\mu$, the above rules exactly determine a probability $p_\lambda(n,d,1)$ with which she would continue guessing her current door, or guess a new door instead with probability $q_\lambda(n,d,1)=1-p_\lambda(n,d,1)$, selected uniformly from the so far unguessed doors.
For example, if $n=2$ and $d=2$, and in the first round she has found a treasure behind door $1$, then she would again guess door $1$ with probability $p_.(2,2,1)=\frac 23$ and guess door $2$ with probability $q_.(2,2,1)=\frac 13$; this is because the conditional expectation for $\mu$ after the first round is that with probability $\frac 23$ both treasures are behind door $1$ and with probability $\frac 13$ there is one treasure behind each door (in this case she had $50\%$ chance to start with door $1$ and not door $2$).
More generally, after finding a treasure in the first round, she would guess a new door with probability $q_.(n,d,1)=\frac{\binom nd}{\binom{n+d-1}d}$, and after this her only option would be to keep on guessing a new door in each round.
Note that $p_\lambda(n,d,1)\to 0$ as $n\to \infty$ for any fixed $\lambda$.

Now we are ready to present the strategy of the searcher for $k>1$ for large enough $n$.
It will again be such that at any time the treasures found behind subsequent doors form a Young diagram $\lambda$.
In each round searcher guesses the last guessed door and guesses $k-1$ new doors with probability $p_\lambda(n,d,k)=k\cdot p_\lambda(n,d,1)$, and otherwise guesses $k$ new doors with probability $q_\lambda(n,d,k)=1-p_\lambda(n,d,k)$.
(Where the new doors are always selected uniformly at random from the so far unguessed doors.)
For new doors to be always available, we need that $n\ge dk$ (as conjectured by Cs\'oka), but for $p_\lambda(n,d,k)$ to denote a probability (i.e., for $p_\lambda(n,d,k)\le 1$), it might be needed that $n$ is even larger; e.g., if $n=6$, $d=3$ and $k=2$, we would get $p_.(6,3,2)>1$ - for a detailed discussion of small values of $n$, see Section \ref{sec:small}.

To finish the proof, we need to show that the above strategy indeed finds any allocation $\mu$ of the treasures with probability $\frac{k^d}{\binom{n+d-1}d}$.
Note that if at anytime the guess contains more than one door that hides undiscovered treasures, searcher will eventually lose.
Thus, it is enough to show that for each $\lambda$ the strategy for general $k$ has exactly $k$ times as much chance of guessing any door $i$ in the next round as the strategy for $k=1$.
It is important that here we consider only how this probability depends on $\lambda$, i.e., for $k>1$ the doors that were guessed earlier but were not revealed to hide a treasure are ignored and treated as not yet guessed.

There are two cases to consider.
The first case is if door $i$ is where searcher has last found a treasure.
In this case searcher's probability of guessing it again is $p_\lambda(n,d,k)=k\cdot p_\lambda(n,d,1)$. 
The second case is if door $i$ is a new, yet unguessed door.
In this case searcher's probability of guessing door $i$ equals the expected number of newly guessed doors divided by the number of doors that are not yet known to contain a treasure (including the already guessed doors that were not revealed to hide a treasure).
The denominator is the same for $k>1$ and $k=1$, thus it is enough to consider the nominator.
According to our strategy, the expected number of newly guessed doors equals $k-1+1-k\cdot p_\lambda(n,d,1)=k\cdot q_\lambda(n,d,1)$.
This finishes the proof.
\end{proof}

\begin{remark}
One could also give the equation for each $\lambda$ that the probabilities need to satisfy so that the probability of finding each $\lambda$-shaped treasure allocations is $\frac{k^d}{\binom{n+d-1}d}$.
For example, the probability of finding a $(2,1)$-shaped allocation is $\frac kn \cdot p_.(n,3,k)\cdot \frac{(k-p_:(n,3,k))}{n-1}=\frac{k\cdot p_.(n,3,k)\cdot (k\cdot q_:(n,3,k))}{n(n-1)}$.
%Since for each allocation $\lambda$ these probabilities must equal $\frac{k^d}{\binom{n+d-1}d}$,
This way we get $p(d)$ equations, where $p(d)$ denotes the {\em partition function}, which equals the number of Young diagrams of size $d$.
Note that since the strategy always finds an allocation and no allocation is found in two different ways, one of these equations is redundant.
The number of variables ($p_\lambda(n,d,k)$'s) of these equations equals the number of Young diagrams of size at most $d-1$ with the additional property that the smallest part doesn't equal the second smallest part.
This number is exactly $p(d)-1$, the number of equations minus one redundant; this is because we can consider each event associated to $p_\lambda(n,d,k)$ as a choice, and every sequence of choices returns a different Young diagram.\footnote{I am thankful to P\'eter Cs\'ikv\'ari for proving this identity using generating functions before I've realized this simple equivalence.}
Although we could not prove that there are no additional redundant equalities and the equations are non-linear, this suggests that there is only one choice for the $p_\lambda(n,d,k)$ (that follows our strategy).
\end{remark}

%\section{Optimal strategies}\label{sec:opt}
%
%Here we discuss some properties that the optimal strategy of the searcher must satisfy in the \sgame and in the \mgame.
%
%\begin{claim} If $n\ge dk$, then searcher should never guess again a previously guessed door that was not revealed to contain a treasure. INAKBB EZ LESZ A COROLLARY?
%\end{claim}
%
%\begin{remark} Not true about doors that contained treasure, like if 6,3,2 and distributed as 2 and 1. ??
%\end{remark}

%If hider reveals at anytime the additional information that a door has no treasures, he helps searcher. TENYLEG BARMELYIK AJTORA IGAZ LENNE? ENNEK NINCS ERTELME IGY. AZ IGAZ, HOGY EDDIGI OLYAN AJTOK KOZUL, AMIKBEN ME'G NEM VOLT KINCS MEGMUTATHAT EGYET ES TOBBIT PERMUTALHATJA? VAGY LEGALABB AKKOR, HA MIND CSAK DISZJUNKT KERDESBEN VOLT?

\section{Small \texorpdfstring{$n$}{n}}\label{sec:small}

As can be seen from the proof of the upper bound of Theorem \ref{thm:s}, any optimal searching strategy that attains the upper bound must guess $k$ new doors in each round.
This also implies that $v_1(n,d,k)<\frac{k^d}{\binom nd}$ if $n< dk$.

The situation for the \mgame is different.

If $d=1$, then the bound in Theorem \ref{thm:m} is sharp for every $n\ge k$.

If $d=2$, the bounds are still always sharp if $n\ge dk$.
For $n\le 2$, this is trivial.
For $n\ge 4$, it can be checked that $p_\lambda(n,2,k)=k\cdot p_\lambda(n,2,1)\le 1$ always holds.
For $n=3$ and $k=1$ or $k=3$, the bounds are again trivially sharp.
Finally, the only interesting case left is if $n=3$ and $k=2$; here we get $p_.(3,2,2)=2\cdot p_.(3,2,2)=1$ and thus $q_.(3,2,2)=0$ - the only possible value for which the algorithm works, as $p_.(3,2,2)\le 1$ indeed denotes a probability and $q_.(3,2,2)=0$ means that we never have to guess the (non-existent) fourth door; a coincidence?
(The same thing happens for these values in \agame, and was already observed by Cs\'oka.)

If $d=3$, the first non-trivial case is $k=2$ and $n=3$.
The uniform hiding strategy gives $v_M(3,3,2)\le \frac{2^3}{\binom{3+3-1}3}=80\%$.
This, however, is suboptimal, as hiding all three treasures behind the same door (chosen uniformly at random) gives $v_M(3,3,2)\le \frac 23$.
(It is easy to see that in fact $v_M(3,3,2)= \frac 23$.)
In general, hiding all treasures behind the same door gives the following upper bound.

\begin{claim}\label{claim:allinone} $v_M(n,d,k)\le \frac kn$.
\end{claim}

It is also interesting to study $d=3$, $k=2$, $n=6$, the first case when $p_.(n,d,k)>1$, as $q_.(6,3,1)=\frac{\binom 63}{\binom{6+3-1}3}=\frac{20}{56}$ and $p_.(6,3,2)=2\cdot p_.(6,3,1)=2\cdot(1-q_.(6,3,1))=2\cdot \frac{36}{56}>1$.
The uniform hiding strategy gives $v_M(6,3,2)\le \frac{2^3}{\binom{6+3-1}3}=\frac 17$, which is better than the bound of Claim \ref{claim:allinone}.
And indeed, $v_M(6,3,2)=\frac 17$ can be achieved by the following searcher strategy.
In the first round, guess two doors, say $1$ and $2$.
Unless searcher loses immediately, one of them, say $1$, is revealed to have a treasure.
In the second round with probability $p_.$ searcher guesses doors $1$ and $3$, and with probability $1-p_.$ doors $3$ and $4$.
If she again finds a treasure behind door $1$, then in round $3$ she guesses doors $1$ and $5$ with probability $p_:$, and doors $5$ and $6$ otherwise.
If she has found the second treasure behind some other door, say, door $3$, then in round $3$ she guesses doors $3$ and $5$ with probability $p_{..}$ and doors $5$ and $6$ otherwise.
A simple calculation gives that $p_. p_:=\frac 37$ and $(2-p_.)(2-p_{..})=\frac {10}7$ are the only conditions these probabilities need to satisfy to find any treasure allocation with probability $\frac 17$.
These probabilities are between $0$ and $1$ if and only if $\frac 47\le p_. \le 1$.

Note that if $p_.=1$, we never need to guess all six doors.
This suggests that the upper bound might be again sharp for $d=3$, $k=2$, $n=dk-1=5$, just like for $d=2$.
And indeed, the choice of $p_.=1$, $p_:=\frac 47$ and $p_{..}=\frac 67$ confirms $v_M(5,3,2)\le \frac{2^3}{\binom{5+3-1}3}=\frac 8{35}$.
In fact, for any $d=3$, $n=3k-1$ values the choice of $p_.=1$, $p_:=\frac{nk^2}{\binom{n+2}3}\to \frac 23$ and $p_{..}=\frac {n+1}{n+2}\to 1$ works.

For $d=3$, $k=2$, $n=4$, the uniform hiding strategy still gives a better bound than Claim \ref{claim:allinone}, but apparently it cannot be attained, as ``there are not enough doors.''
This seems like a good candidate where the different variants of the game might have different values.

%MI VAN, HA mindig 2-1 ARANYBAN REJTJUK EL? Guess two-two different ones, then both: 8 eset a 30-bol. Mutatja valszeg, hogy nem eri meg mindig olyan strategiat, ami nem tippel regi lyukat.

\section{Concluding remarks}\label{sec:remarks}

What happens if $n$ is small?
Do some $p_\lambda(n,d,k)$ probabilities exist for every $n\ge dk$?
A possible approach might be to prove their existence using some argument similar to the one used Baranyai's theorem \cite{Baranyai}.
Can it ever be useful to pick less than $k$ doors?
We could not eliminate this possibility for $n<dk$.
Does it matter whether get a random treasure or hider reveals one?
What if instead a beneficent helper of searcher reveals the treasures?
Does $v_A(n,d,k)=v_M(n,d,k)$ always hold?

What happens if we need to find only $d'$ of the $d$ treasures?

What happens if we can make $e$ unsuccessful guesses?

What happens if instead of bounded query size we impose other restrictions on the queries in caching games of this type?

\subsection*{Acknowledgment}

I am thankful to Endre Cs\'oka for discussions on the topic.

\renewcommand{\refname}

\end{document}